\documentclass{amsart}
\usepackage{amsthm, xcolor, bm}
\usepackage{latexsym, amssymb, amsfonts, amsmath, graphicx, url, mathtools}
\usepackage[mathscr]{eucal}
\usepackage[vcentermath,enableskew]{youngtab}
\usepackage{color}
\usepackage[noadjust]{cite} 
\usepackage{tikz,hyperref}
\usepackage[margin=1in]{geometry}
\usepackage{enumerate,textcomp, wasysym, bbm}

\newtheorem{theorem}{Theorem}[section]
\newtheorem{proposition}[theorem]{Proposition}
\newtheorem{lemma}[theorem]{Lemma}
\newtheorem{conjecture}[theorem]{Conjecture}
\newtheorem{corollary}[theorem]{Corollary}

\theoremstyle{definition}
\newtheorem{definition}[theorem]{Definition}
\newtheorem{example}[theorem]{Example}
\newtheorem{remark}[theorem]{Remark}

\newtheoremstyle{named}{}{}{\itshape}{}{\bfseries}{.}{.5em}{#1 \thmnote{#3}}
\theoremstyle{named}

\newtheorem*{namedconjecture}{Conjecture}

\newcommand{\B}{\mathcal{B}}

\DeclareMathOperator{\supp}{supp}
\DeclareMathOperator{\rajcode}{rajcode}

\title{On the Support of Grothendieck Polynomials}

\author{Karola M\'esz\'aros}
\address{Karola M\'esz\'aros, Department of Mathematics, Cornell University, Ithaca, NY 14853. \newline\textup{karola@math.cornell.edu}
}

\author{Linus Setiabrata}
\address{Linus Setiabrata, Department of Mathematics, University of Chicago, Chicago, IL, 60637. \newline\textup{linus@math.uchicago.edu}
}

\author{Avery St.~Dizier}
\address{Avery St. Dizier, Department of Mathematics, University of Illinois at Urbana-Champaign, Urbana, IL 61801. \newline\textup{stdizie2@illinois.edu}
}

\thanks{Karola M\'esz\'aros received support from CAREER NSF Grant DMS-1847284. Avery St.~Dizier received support from NSF Grant DMS-2002079.}

\begin{document}

\begin{abstract}
	Grothendieck polynomials $\mathfrak{G}_w$ of permutations $w\in S_n$ were introduced by Lascoux and Sch\"utzenberger in 1982 as a set of distinguished representatives for the K-theoretic classes of Schubert cycles in the\ K-theory of the flag variety of $\mathbb{C}^n$. We conjecture that the exponents of nonzero terms of the Grothendieck polynomial $\mathfrak{G}_w$ form a poset under componentwise comparison that is isomorphic to an induced subposet 
	of $\mathbb{Z}^n$. When $w\in S_n$ avoids a certain set of patterns, we conjecturally connect the coefficients of $\mathfrak{G}_w$ with the M\"obius function values of the aforementioned poset with $\hat{0}$ appended. We prove special cases of our conjectures for Grassmannian and fireworks permutations.  		
\end{abstract}
	\maketitle
	
	\section{Introduction}
	
	Grothendieck polynomials $\mathfrak{G}_w$ are multivariate polynomials associated to permutations $w\in S_n$. Grothendieck polynomials were introduced by Lascoux and Sch\"utzenberger in \cite{LS2} as a set of distinguished representatives for the K-theoretic classes of Schubert cycles in the\ K-theory of the flag variety of $\mathbb{C}^n$. The lowest degree component of $\mathfrak G_w$ is the Schubert polynomial $\mathfrak S_w$. Schubert polynomials have many combinatorial constructions and are well-understood \cite{laddermoves, BJS, FKschub, nilcoxeter, thomas, lenart, prismtableaux,balancedtableaux,FMS,schlorentzian,schubertgt,schubvertices, inclusionexclusion}. However there is not nearly as much known combinatorially or discrete-geometrically about Grothendieck polynomials. A recent flurry of work \cite{MTY,grothlattice,grothice,grothcoloredlattice,PSW,grothortho,symmgrothnewton} has uncovered novel formulas and perspectives on Grothendieck polynomials.
	
	The main objective of this paper is to shed light on the combinatorial structure of the support of Grothendieck polynomials. While there have been recent breakthroughs in the degree of Grothendieck polynomials \cite{symmgrothdeg, grothortho, PSW}, much less is known about the structure of the support. The support has previously been conjecturally connected to generalized permutahedra via flow polytopes \cite[Conjecture 5.1]{genpermflows}, and via the Lorentzian property \cite[Conjecture 22]{schlorentzian}. In this paper, we give a new poset theoretic perspective on the support of any Grothendieck polynomial. 
	
	\subsection{Supports of Grothendieck polynomials} 
	For a permutation $w\in S_n$, the support of $\mathfrak{G}_w$ is the set of exponents of terms in $\mathfrak{G}_w$ with nonzero coefficient. We endow the support $\supp(\mathfrak{G}_w)$ with the following poset structure.
	For $\alpha, \beta\in\mathbb{Z}^n$, define the \emph{componentwise comparison} $\leq$ by 
	$\alpha\leq \beta \mbox{ if } \alpha_i\leq \beta_i \mbox{ for all } i\in [n].$	
	\begin{figure}[ht]
		\includegraphics[scale=1]{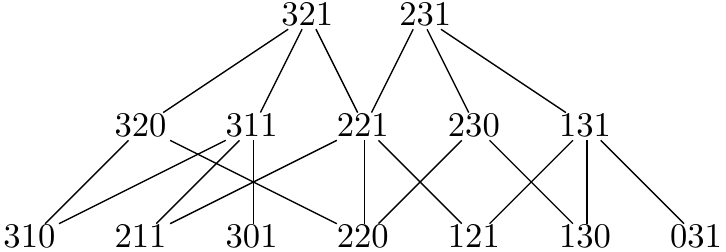}
		\caption{The Hasse diagram of $\supp(\mathfrak{G}_{15324})$ under componentwise comparison (writing exponents $(a,b,c,0,0)\in\mathbb{Z}^5$ as $abc$).}
		\label{fig:poset-example1}
	\end{figure}
	
	For $w=15324$, we have 
	\begin{align*}
	\mathfrak{G}_w &= (x_1^3x_2 + x_1^3x_3 + x_1^2x_2^2 + x_1^2 x_2 x_3 + x_1x_2^3 + x_1 x_2^2 x_3 +x_2^3 x_3)\\
				  &\phantom{aa}-(x_1^3 x_2^2+2x_1^3x_2 x_3 + x_1^2x_2^3 +2x_1^2x_2^2 x_3 + 2x_1x_2^3x_3)\\
				  &\phantom{aa}+ (x_1^3x_2^2 x_3 + x_1^2 x_2^3 x_3).
	\end{align*}
	The Hasse diagram of $\supp(\mathfrak{G}_w)$ as a poset under componentwise comparison is shown in Figure \ref{fig:poset-example1}.
	
	\begin{conjecture}
		\label{conj:1}
		If $\alpha\in \mathrm{supp}(\mathfrak{G}_w)$ and $|\alpha|<\deg \mathfrak{G}_w$, then there exists $\beta\in\mathrm{supp}(\mathfrak{G}_w)$ with $\alpha< \beta$.
	\end{conjecture}
	
	We prove Conjecture \ref{conj:1} for fireworks permutations in Theorem \ref{thm:fireworks}. A natural strengthening of Conjecture \ref{conj:1} is:	
	\begin{conjecture}
		\label{conj:2}
		If $\alpha\in \mathrm{supp}(\mathfrak{G}_w)$ and $|\alpha|<\deg \mathfrak{G}_w$, then there exists $\beta\in\mathrm{supp}(\mathfrak{G}_w)$ with $\alpha< \beta$ and $|\beta|=|\alpha|+1$.
	\end{conjecture}
	
	We also conjecture that $\supp(\mathfrak{G}_w)$ is closed under taking intervals in componentwise comparison.
	\begin{conjecture}
		\label{conj:3}
		Fix any $w\in S_n$. If $\alpha,\gamma\in\mathrm{supp}(\mathfrak{G}_w)$, then
		\[\left\{\beta\in \mathbb{Z}^n \mid \alpha\leq \beta\leq \gamma \right\} \subseteq \mathrm{supp}(\mathfrak{G}_w).  \]
	\end{conjecture}
	
	A discrete-geometric strengthening of Conjecture \ref{conj:3} is: 
	
	\begin{conjecture}
		\label{conj:4}
		For all $w\in S_n$, $\mathfrak{G}_w$ has SNP and $\mathrm{Newton}(\mathfrak{G}_w)$ is a generalized polymatroid.
	\end{conjecture}
	We prove Conjectures \ref{conj:1}--\ref{conj:4} for Grassmannian permutations in Theorem \ref{thm:grass}. In Theorem \ref{thm:superset}, we provide a polytope containing the Newton polytope $\mathrm{Newton}(\mathfrak{G}_w)$. Assuming Conjectures \ref{conj:1} and \ref{conj:4}, we characterize when equality with the Newton polytope occurs in Proposition \ref{prop:converse}.
	
	\subsection{Coefficients of Grothendieck polynomials} 
	In Definition \ref{def:poset}, we define a poset $P_w\subseteq \mathbb{Z}^n$ (under componentwise comparison) containing $\supp(\mathfrak{G}_w)$. It appears for permutations whose Schubert polynomial $\mathfrak{S}_w$ has all nonzero coefficients equal 1, that the coefficients of $\mathfrak{G}_w$ agree with the M\"obius function of $P_w$:
	
	\begin{conjecture} 
		\label{conj:mobius} 
		Let $w$ be a permutation such that all nonzero coefficients of $\mathfrak{S}_w$ equal 1. If $\mu_w$ is the M\"obius function of $P_w$, then
		\[\mathfrak{G}_w=-\sum_{\alpha\in P_w-\hat{0}}\mu_w(\hat{0}, \alpha)x^{\alpha}.\]
	\end{conjecture}
	It is known that all nonzero coefficients of $\mathfrak{S}_w$ equal 1 exactly when $w$ avoids the patterns 12543, 13254, 13524, 13542, 21543, 125364, 125634, 215364, 215634, 315264, 315624, and 315642 (\cite[Theorem 4.8]{zeroone}). We conjecture one final property connecting the poset structure of $\supp(\mathfrak{G}_w)$ to the coefficients of $\mathfrak{G}_w$:
	
	\begin{conjecture} 
		\label{conj:coeff}
		Fix $w\in S_n$ and let $\mathfrak{G}_w = \sum_{\alpha\in\mathbb{Z}^n}C_{w\alpha}x^\alpha$. For any $\beta\in\supp(\mathfrak{G}_w^{\mathrm{top}})$,
		\[\sum_{\alpha\leq \beta} C_{w\alpha} = 1. \]
	\end{conjecture}
	When the poset $\supp(\mathfrak{G}_w)$ has a unique maximum element, 
	Conjecture \ref{conj:coeff} coincides with the principal specialization $\mathfrak{G}_w(1,\ldots,1)$. While we could not find a proof in the literature, it is known that $\mathfrak{G}_w(1,\ldots,1)=1$ (see for instance \cite[Comment 3.2]{dunkl}). We provide a short proof in Proposition \ref{prop:grothprinspec}, concluding a special case of Conjecture \ref{conj:coeff}. We have tested Conjectures \ref{conj:1}--\ref{conj:coeff} for all $w\in S_8$. We note that all our conjectures naturally lift to double Grothendieck polynomials, and appear to hold there as well.
	
	 \subsection{Outline of the paper} Section \ref{sec:back} covers the necessary background for the paper. In Section \ref{sec:support} we elaborate on Conjectures \ref{conj:1}--\ref{conj:4} and prove related results on the support of Grothendieck polynomials. In Section \ref{sec:mobius} we conjecturally connect the coefficients of certain Grothendieck polynomials to the M\"obius function of a poset. We conclude by considering the principal specialization of Grothendieck polynomials.
	
	\section{Background}
	\label{sec:back}
	\subsection{Conventions}
	For $n\in \mathbb{N}$, we use the notation $[n]$ to mean the set $\{1,2,\ldots,n \}$. We reserve lowercase Greek letters $\alpha,\beta,\gamma,\delta,\epsilon$ for nonnegative integer vectors in $\mathbb{R}^n$; we opt for $t$ to denote arbitrary vectors in $\mathbb{R}^n$. We write $|\alpha|$ for $\alpha_1+\cdots+\alpha_n$. We use $x$ to represent the collection of variables $x_1,x_2,\ldots,x_n$, so $x^\alpha$ denotes the monomial $x_1^{\alpha_1}\cdots x_n^{\alpha_n}$.
	
	For $j\in [n-1]$, $s_j$ will denote the adjacent transposition in the symmetric group $S_n$ swapping $j$ and $j+1$. We otherwise represent permutations $w\in S_n$ in one-line notation as a word $w(1)w(2)\cdots w(n)$, so $w=312\in S_3$ is the permutation that sends $1\mapsto 3$, $2\mapsto 1$, and $3\mapsto 2$. Throughout, we will take permutations as acting on the right (switching positions, not values). For example $ws_1$ equals $w$ with the numbers $w(1)$ and $w(2)$ swapped. We write $\ell(w)$ for the number of inversions of $w$.
	
	\subsection{Schubert and Grothendieck Polynomials}
	
	\begin{definition}
		Fix any $n\geq 1$. The \emph{divided difference operators} $\partial_j$ for $j\in[n-1]$ are operators on the polynomial ring $\mathbb{Z}[x_1,\ldots,x_n]$ defined by
		\[\partial_j(f)=\frac{f-(s_j\cdot f)}{x_j-x_{j+1}}
		=\frac{f(x_1,\ldots,x_n)-f(x_1,\ldots,x_{j-1},x_{j+1},x_j,x_{j+2},\ldots,x_n)}{x_j-x_{j+1}}.\]
		The \emph{isobaric divided difference operators} $\overline{\partial}_j$ are defined on $\mathbb{Z}[x_1,\ldots,x_n]$ by
		\begin{align*}
		\overline{\partial}_j(f)&=\partial_j((1-x_{j+1})f).
		\end{align*}
	\end{definition}
	
	\begin{definition}
		The \emph{Schubert polynomial} $\mathfrak{S}_w$ of $w\in S_n$ is defined recursively on the weak Bruhat order. Let $w_0=n \hspace{.1cm} n\!-\!1 \hspace{.1cm} \cdots \hspace{.1cm} 2 \hspace{.1cm} 1 \in S_n$, the longest permutation in $S_n$. If $w\neq w_0$ then there is $j\in [n-1]$ with $w(j)<w(j+1)$ (called an \emph{ascent} of $w$). The polynomial $\mathfrak{S}_w$ is defined by
		\begin{align*}
		\mathfrak{S}_w=\begin{cases}
		x_1^{n-1}x_2^{n-2}\cdots x_{n-1}&\mbox{ if } w=w_0,\\
		\partial_j \mathfrak{S}_{ws_j} &\mbox{ if } w(j)<w(j+1).
		\end{cases}
		\end{align*}
	\end{definition}
	
	\begin{definition}
		The \emph{Grothendieck polynomial} $\mathfrak{G}_w$ of $w\in S_n$ is defined analogously to the Schubert polynomial,
		with
		\begin{align*}
		\mathfrak{G}_w=\begin{cases}
		x_1^{n-1}x_2^{n-2}\cdots x_{n-1}&\mbox{ if } w=w_0,\\
		\overline{\partial}_j \mathfrak{G}_{ws_j} &\mbox{ if } w(j)<w(j+1).
		\end{cases}
		\end{align*}	
	\end{definition}

	It can be seen from the recursive definitions that $\mathfrak{S}_w$ is homogeneous of degree equal to $\ell(w)$, and equals the lowest-degree nonzero homogeneous component of $\mathfrak{G}_w$. See \cite{manivel} for a deeper introduction to Schubert polynomials. We now recall pipe dreams, one of many combinatorial constructions of Schubert and Grothendieck polynomials.
	
	\begin{definition}
		A \emph{pipe dream} for $w\in S_n$ is a tiling of an $n\times n$ matrix with crosses \includegraphics[scale=.13]{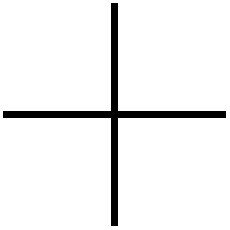} and elbows \includegraphics[scale=.13]{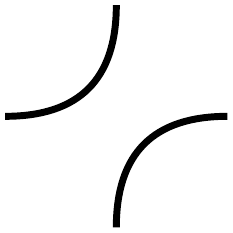} such that
		\begin{itemize}
			\item All tiles in the weak south-east triangle are elbows, and
			\item If you write $1,2,\ldots, n$ on the top left-to-right and follow the strands (treating second crossings among the same strands as elbows instead), they come out on the left edge and read $w$ from top to bottom.
		\end{itemize}
		A pipe dream is \emph{reduced} if no two strands cross twice. Let $\mathrm{RPD}(w)$ and $\mathrm{PD}(w)$ denote respectively the sets of reduced pipe dreams of $w$ and all pipe dreams of $w$.
	\end{definition}
	
	\begin{theorem}[{\cite{FKschub,laddermoves,FKgroth}}]
		\label{thm:pdformula}
		For any permutation $w\in S_n$,
		\[\mathfrak{S}_w = \sum_{P\in \mathrm{RPD}(w)}{x^{\mathrm{wt}(P)}} \quad \mbox{and} \quad \mathfrak{G}_w = \sum_{P\in \mathrm{PD}(w)}{(-1)^{\#\mathrm{crosses}(P)-\ell(w)}x^{\mathrm{wt}(P)}} \]
		where $\mathrm{wt}(P)_i=\mbox{\# crosses in row i of }P$.
	\end{theorem}
	
	\begin{example}
		The pipe dreams of $w=1432$ are shown in Figure \ref{fig:pipe-dreams-1432}.
	\end{example}

	\begin{figure}[ht]
		\begin{center}
			\includegraphics{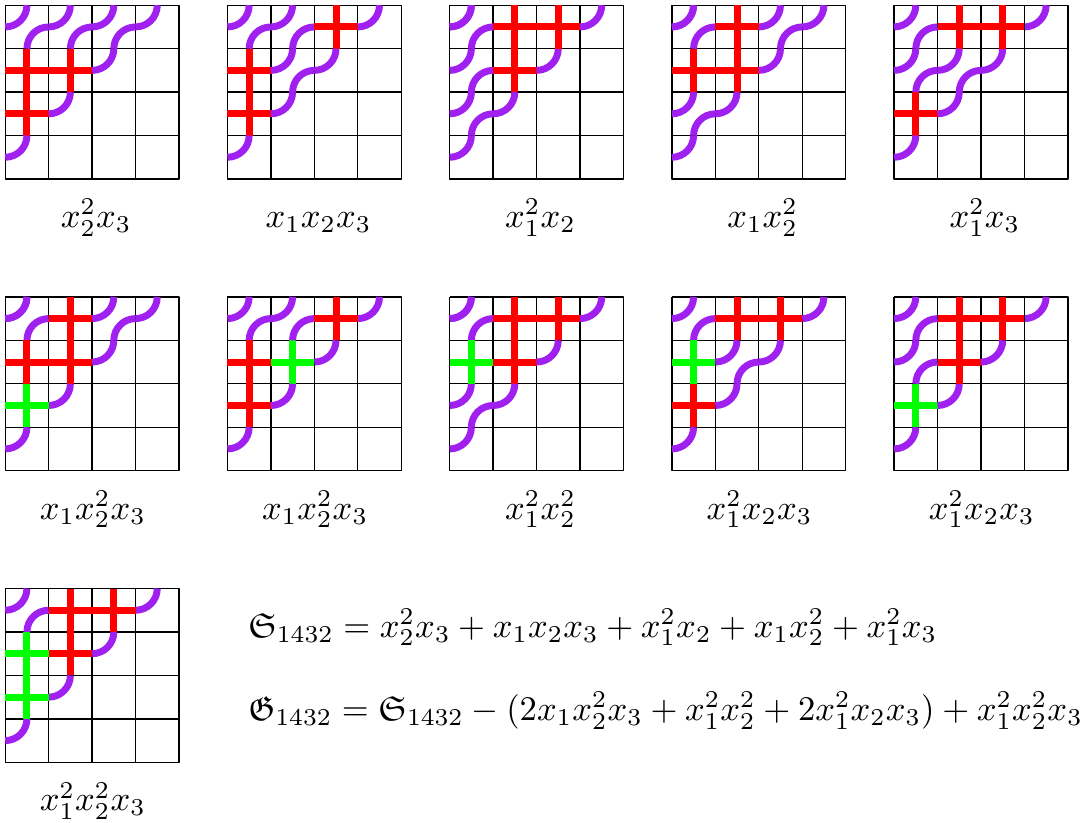}
		\end{center}
		\caption{The set $\mathrm{PD}(1432)$. Second crossing of strands are shown in green.}
		\label{fig:pipe-dreams-1432}
	\end{figure}
	
	Pipe dreams also carry additional geometric structure, which we utilize in Section \ref{sec:mobius}.
	\begin{theorem}[\cite{KM}]
		\label{thm:pdcomplex}
		For any $w\in S_n$, the reduced pipe dreams $\mathrm{RPD}(w)$ equal the set of facets of a pure simplicial complex $\Delta_w$ of dimension $\binom{n}{2}-\ell(w)-1$. The interior faces of $\Delta_w$ are exactly $\mathrm{PD}(w)$. The boundary of $\Delta_w$ is the union of all complexes $\Delta_v$ where $v\geq w$ (in strong Bruhat order).
	\end{theorem}
	
	\begin{theorem}[{\cite[Corollary 3.8]{KM}}]
		\label{thm:ball}
		The simplicial complex $\Delta_w$ of $w\in S_n$ is a ball whenever $w\neq w_0$, and $\Delta_{w_0}$ is empty.
	\end{theorem}
	
	\begin{example}
		\label{exp:pdcomplex}
		When $w=1432$, $\Delta_w$ is shown in Figure \ref{fig:pdcomplex1432} (with faces labeled by their pipe dream).
	\end{example}
	
	\begin{figure}[ht]
		\begin{center}
			\includegraphics[scale=.52]{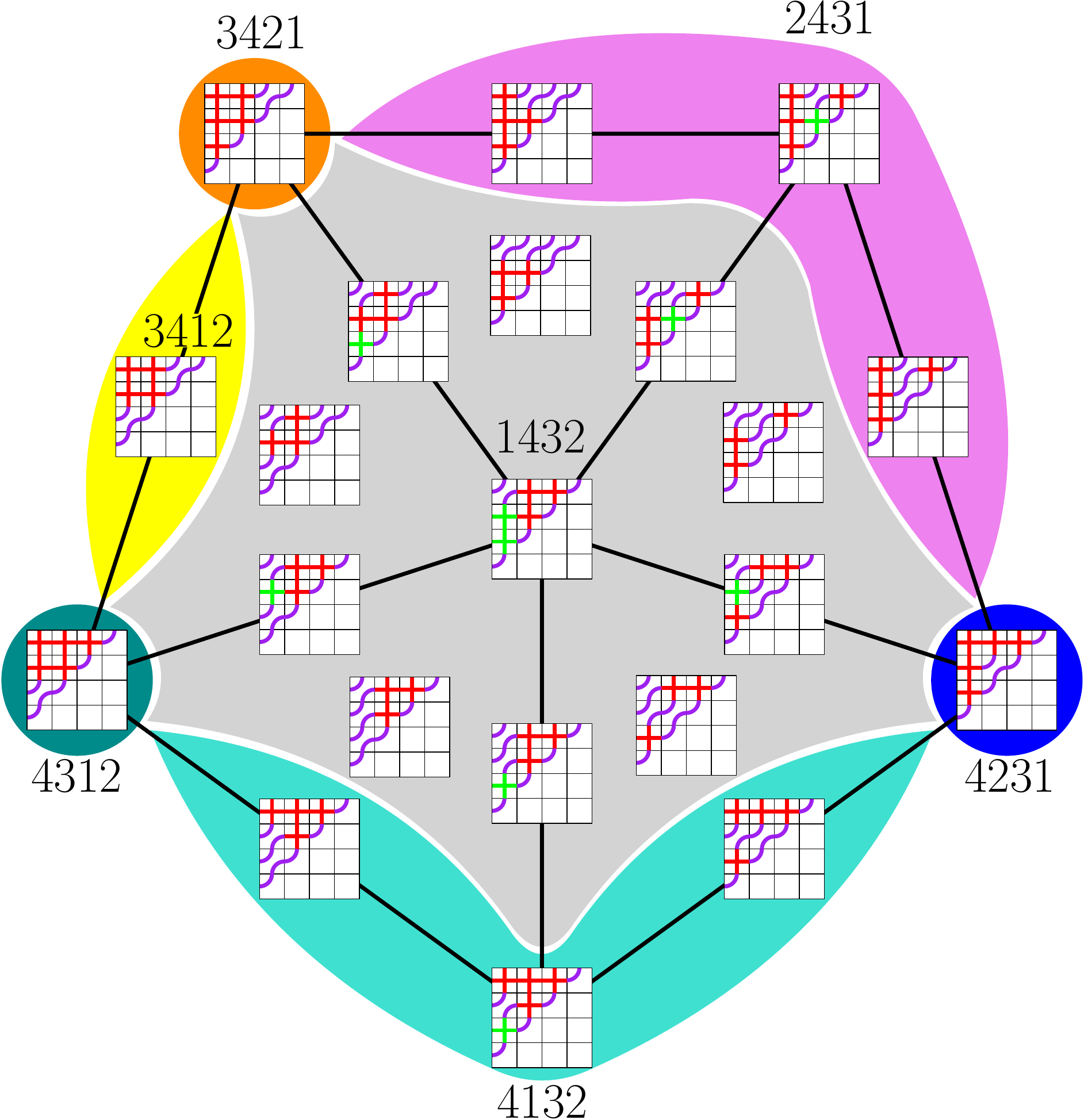}
			\caption{The simplicial complex $\Delta_{1432}$. The colored regions indicate pipe dreams corresponding to a particular permutation.}
			\label{fig:pdcomplex1432}
		\end{center}
	\end{figure}
	
	\subsection{Support and Newton Polytopes of Polynomials}
	
	The \emph{support} of a polynomial $f=\sum_{\alpha\in \mathbb{Z}^n}c_\alpha x^{\alpha} \in \mathbb{R}[x_1,\ldots,x_n]$ is the set
	\[\supp(f)=\{\alpha \mid c_{\alpha} \neq 0\}\subset \mathbb{Z}^n. \]
	The \emph{Newton polytope} of $f$, denoted $\mathrm{Newton}(f)$, is the convex hull of $\supp(f)$. When \[\mathrm{Newton}(f)\cap \mathbb{Z}^n = \supp(f),\] 
	$f$ is said to have \emph{saturated Newton polytope}, abbreviated SNP.
	
	\subsection{Generalized Permutahedra and Generalized Polymatroids}
	
	A function $z:2^{[n]}\to \mathbb{R}$ is called \emph{submodular} if 
	\[z(I)+z(J)\geq z(I\cup J)+z(I\cap J) \mbox{ for all } I,J\subseteq [n]. \]
	Similarly, $z$ is called \emph{supermodular} if $-z$ is submodular.
	\begin{definition}
		A polytope $P\subset \mathbb{R}^n$ is a \emph{generalized permutahedron} if there is a submodular function $z$ such that 
		\[P=\left\{t\in \mathbb{R}^n \ \middle|\ \sum_{i\in I} t_i\leq z(I) \mbox{ for all } I\subseteq [n] \mbox{ and } \sum_{i=1}^n t_i= z([n]) \right\}. \]
	\end{definition}
	
	\begin{definition}
		A pair $(y,z)$ of functions $2^{[n]}\to \mathbb{R}$ is called a \emph{paramodular pair} if $y$ is supermodular, $z$ is submodular, and
		\[z(I)-y(J)\geq z(I\setminus J)-y(J\setminus I) \mbox{ for all } I,J\subseteq [n]. \]
	\end{definition}
	
	\begin{definition}
		\label{def:gpolymatroid}
		A polytope $Q\subset \mathbb{R}^n$ is called a \emph{generalized polymatroid} if there is a paramodular pair $(y,z)$ such that 
		\[Q=\left\{t\in \mathbb{R}^n \ \middle| \ y(I)\leq \sum_{i\in I} t_i\leq z(I) \mbox{ for all } I\subseteq [n]\right\}. \]
	\end{definition}

	Generalized permutahedra are special cases of generalized polymatroids.
	
	\begin{lemma}[{\cite[Theorem 14.2.8]{frank}}]
		\label{lem:recovery}
		Let $Q\subseteq \mathbb{R}^n$ be a generalized polymatroid with paramodular pair $(y,z)$. Then $y$ and $z$ are uniquely determined from $Q$ as
		\[y_I = \min\left\{\sum_{i\in I} q_{i} \ \middle| \ q\in Q\right\} \quad \mbox{and}\quad z_I = \max\left\{\sum_{i\in I} q_{i} \ \middle| \ q\in Q\right\}. \]
	\end{lemma}
	
	\begin{lemma}[{\cite[Theorem 14.2.10]{frank}}]
		\label{lem:integralgpoly}
		If $Q$ is a generalized polymatroid defined by an integral paramodular pair $(y,z)$, then $Q$ is an integral polyhedron. Furthermore, there are always integral optimizers for 
		\[
		\min\left\{\sum_{i\in I} q_{i} \ \middle| \ q\in Q\right\} \quad \mbox{and}\quad \max\left\{\sum_{i\in I} q_{i} \ \middle| \ q\in Q\right\}. 
		\]
	\end{lemma}

	The following proposition is immediate from \cite[Theorem 1]{recognizing}.
	\begin{proposition}
		\label{prop:mink-sum}
		If $Q,Q'\subset \mathbb{R}^n$ are generalized polymatroids, then so is $Q+Q'$.
	\end{proposition}
	
	\subsection{Diagrams}
	
	By a \emph{diagram}, we mean a sequence 
	\[D=(D_1,D_2,\ldots,D_n)\]
	of finite subsets of $[n]$, called the \emph{columns} of $D$. We interchangeably think of $D$ as a collection of boxes $(i,j)$ in a grid, viewing an element $i\in D_j$ as a box in row $i$ and column $j$ of the grid. When we draw diagrams, we read the indices as in a matrix: $i$ increases top-to-bottom and $j$ increases left-to-right. Associated to any permutation $w\in S_m$ is the \emph{Rothe diagram} $D(w)$, defined by
	\[ D(w)=\{(i,j)\in [n]\times [n] \mid i<w^{-1}(j)\mbox{ and } j<w(i) \}. \]
	
	For $R,S\subseteq [n]$, we write $R\preccurlyeq S$ if $\#R=\#S$ and the $k$\/th least element of $R$ does not exceed the $k$\/th least element of $S$ for each $k$. For any diagrams $C=(C_1,\ldots, C_n)$ and $D=(D_1,\ldots, D_n)$, we say $C\preccurlyeq D$ if $C_j\preccurlyeq D_j$ for all $j\in[n]$. 
	
	\subsection{Matroids and Polytopes}
	
	A \emph{matroid} $M$ is a pair $(E, \B)$ consisting of a finite set $E$ and a nonempty collection of subsets $\B$ of $E$, 
	called the \emph{bases} of $M$. $\B$ is required to satisfy the \emph{basis exchange axiom}: 
	If $B_1, B_2 \in \B$ and $b_1 \in B_1- B_2$, then there exists $b_2 \in B_2 - B_1$ such that $B_1 - b_1 \cup b_2 \in \B$. By choosing a labeling of the elements of $E$, we always assume $E=[n]$ for some $n$.
	
	Given a matroid $M=(E,\mathcal{B})$ with $E=[n]$ and a basis $B\in\mathcal{B}$, let $\zeta^B$ be the indicator vector of $B$. That is, let $\zeta^B=(\zeta_1^B,\ldots, \zeta_n^B)\in\mathbb{R}^n$ with $\zeta_i^B=1$ if $i\in B$ and $\zeta_i^B=0$ if $i\notin B$ for each $i$. The \emph{matroid polytope} of $M$ is the polytope 
	\[P(M)=\mathrm{Conv}\{\zeta^B:\,B\in\mathcal{B} \}.\]
	
	Any set $S\subseteq [n]$ with $S\supseteq B$ for $B\in\mathcal{B}$ is called a \emph{spanning set} of $M$.
	The \emph{spanning set polytope} $P_{\mathrm{sp}}(M)$ is the polytope
	\[P_\mathrm{sp}(M)=\mathrm{Conv}\{\zeta^S\ \mid \ S\subseteq [n] \mbox{ is a spanning set of } M \}.\]
	
	The \emph{rank function} of $M$ is the function
	\[r:2^{E}\to \mathbb{Z}_{\geq 0}\] defined by $r(S)=\max\{\#(S\cap B):\, B\in \mathcal{B} \}$. The sets $S\cap B$ where $S\subseteq [n]$ and $B\in\mathcal{B}$ are called the \emph{independent sets} of $M$. 
	
	The following result is well-known, see for instance \cite{frank, Schrijver}.
	\begin{proposition}
		\label{prop:matroidgps}
		For any matroid $M$ on ground set $[n]$, $P(M)$ is a generalized permutahedron and $P_{\mathrm{sp}}(M)$ is a generalized polymatroid.
	\end{proposition}

	As a generalized permutahedron, a matroid polytope is parameterized by the (submodular) rank function $r$ of the underlying matroid:
	\begin{align*}
	P(M)&=\left\{ {t}\in\mathbb{R}^n \ \middle| \ \sum_{i\in I}{t_i}\leq r(I) \mbox{ for } I\subseteq E, \mbox{ and } \sum_{i\in E}{t_i}=r(E) \right\}.
	\end{align*}
	As a generalized polymatroid, a spanning set polytope is parameterized by
	\begin{align*}
	P_{\mathrm{sp}}(M)&=\left\{ {t}\in\mathbb{R}^n \ \middle| \ r(E)-r(E\setminus I) \leq \sum_{i\in I}{t_i}\leq |I| \mbox{ for all } I\subseteq E \right\}.
	\end{align*}
	
	\begin{definition}
		\label{def:schubert-matroids}
		Fix positive integers $1 \leq s_1 < \cdots < s_r \leq n$. The sets $\{a_1, \ldots, a_r\}$ of positive integers with $a_1<\cdots<a_r$ such that $a_1 \leq s_1, \ldots, a_r \leq s_r$ are the bases of a matroid (with ground set $[n]$), 
		called the \emph{Schubert matroid} $\mathrm{SM}_n(s_1, \ldots, s_r)$.
	\end{definition}
	
	\begin{theorem}[{\cite[Theorem 11]{FMS}}]
		\label{thm:fms}
		For any permutation $w\in S_n$ with Rothe diagram $D(w)=(D_1,\ldots, D_n)$,
		\[\mathrm{Newton}(\mathfrak{S}_w)=\sum_{j=1}^{n}P(\mathrm{SM}_n(D_j)). \]
		In particular, each $\alpha\in\supp(\mathfrak{S}_w)$ can be written as a sum
		\[\alpha = \alpha^{(1)}+\cdots+\alpha^{(n)} \]
		where $\alpha^{(j)}$ is the indicator vector of a basis of $\mathrm{SM}_n(D_j)$.
	\end{theorem}

	\begin{example}
		Consider the permutation $w=21543$. Then
		\begin{center}
			\includegraphics[scale=1]{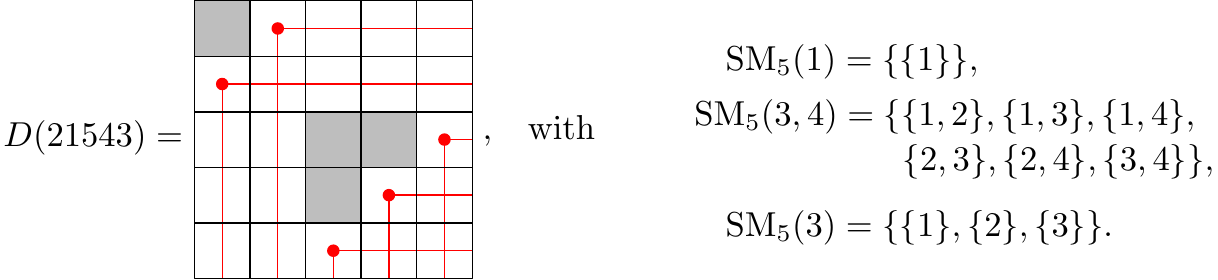}
		\end{center}
	The Minkowski sum decomposition of $\mathrm{Newton}(\mathfrak{S}_{21543})$ is shown in Figure \ref{fig:minkowski-sum}.
	\end{example}

	\begin{figure}[ht]
		\begin{center}
			\includegraphics[scale=.8]{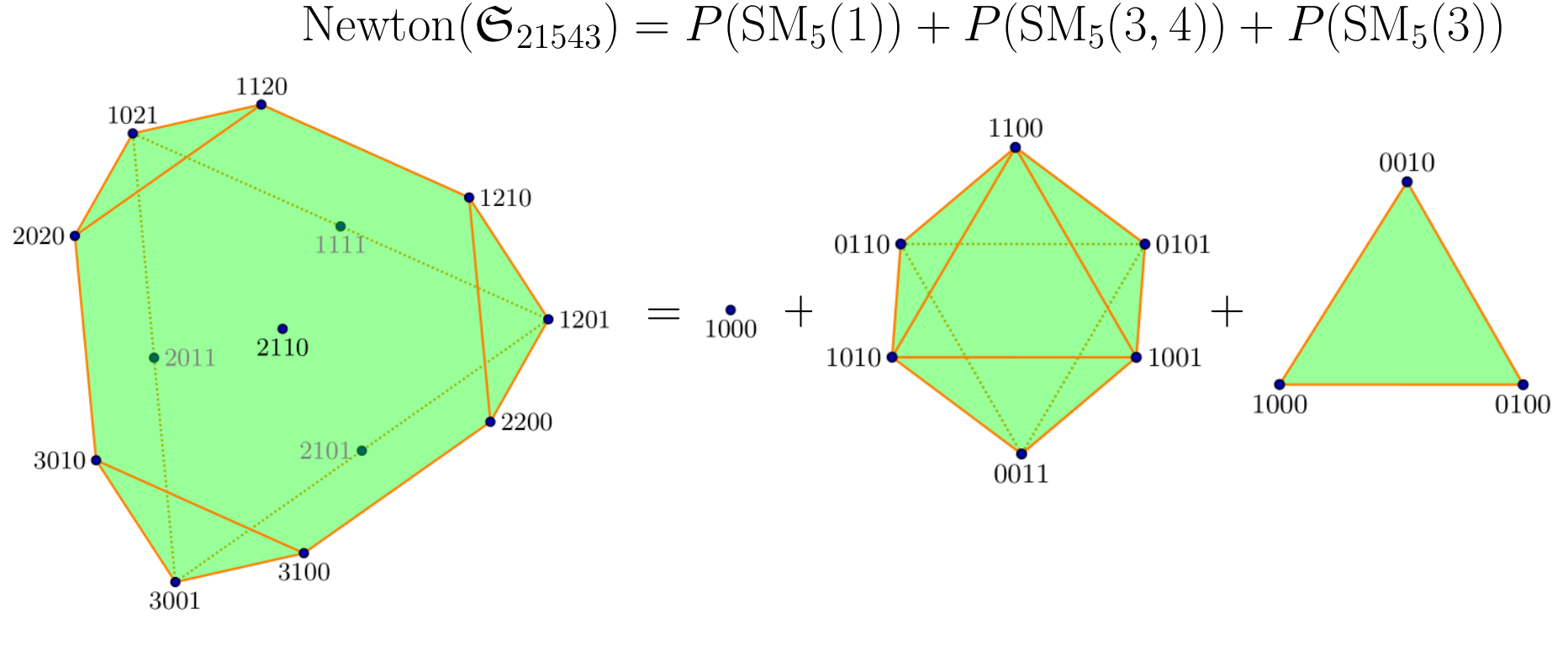}
		\end{center}
	\caption{The Schubert matroid polytope decomposition of $\mathrm{Newton}(\mathfrak{S}_{21543})$.}
	\label{fig:minkowski-sum}
	\end{figure}
	
	\section{Supports of Grothendieck Polynomials}
	\label{sec:support}
	For a permutation $w\in S_n$, recall the support of the Grothendieck polynomial of $w$ is the set $\supp(\mathfrak{G}_w)$ of exponents of terms in $\mathfrak{G}_w$ with nonzero coefficient. We endow $\supp(\mathfrak{G}_w)$ with the following poset structure.
	
	\begin{definition}
		For $\alpha, \beta\in\mathbb{Z}^n$, define the \emph{componentwise comparison} $\leq$ by 
		\[\alpha\leq \beta \mbox{ if } \alpha_i\leq \beta_i \mbox{ for all } i\in [n].\]
	\end{definition}

	We study the subsets $\supp(\mathfrak{G}_w)\subset\mathbb{Z}^n$ with the inherited poset structure.
	
	\begin{example}
		For $w=15324$, we have 
		\begin{center}
			\includegraphics{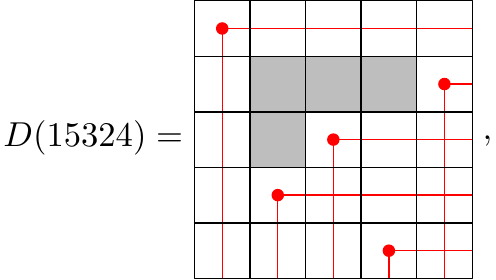}
		\end{center}
		\begin{align*}
		\mathfrak{S}_w &=x_1^3x_2 + x_1^3x_3 + x_1^2x_2^2 + x_1^2 x_2 x_3 + x_1x_2^3 + x_1 x_2^2 x_3 +x_2^3 x_3,\\
		\mathfrak{G}_w &= \mathfrak{S}_w - (x_1^3 x_2^2+2x_1^3x_2 x_3 + x_1^2x_2^3 +2x_1^2x_2^2 x_3 + 2x_1x_2^3x_3) + (x_1^3x_2^2 x_3 + x_1^2 x_2^3 x_3).
		\end{align*}
		The Hasse diagram of $\supp(\mathfrak{G}_w)$ as a poset under componentwise comparison is shown in Figure \ref{fig:poset-example}.
	\end{example}
	
	\begin{figure}[ht]
		\includegraphics[scale=1]{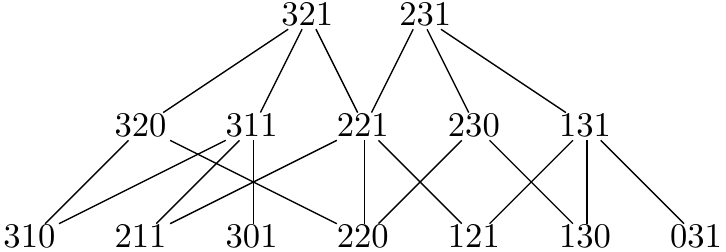}
		\caption{The Hasse diagram of $\supp(\mathfrak{G}_{15324})$ (writing exponents $(a,b,c,0,0)\in\mathbb{Z}^5$ as $abc$).}
		\label{fig:poset-example}
	\end{figure}
	
	We first present two known properties of the posets $\supp(\mathfrak{G}_w)$.
	
	\begin{lemma}
		\label{lem:downwardsdivisibility}
		Fix any permutation $w\in S_n$. For each $\beta\in \supp(\mathfrak{G}_w)$, with $|\beta|>\ell(w)$, there is $\alpha\in\supp(\mathfrak{G}_w)$ with $\alpha\leq\beta$ and $|\alpha|=|\beta|-1$.
	\end{lemma}
	\begin{proof}
		Choose any pipe dream $P$ of $w$ with weight $\beta$. Since $|\beta|>\ell(w)$, $P$ is not reduced. Removing any single second crossing in $P$ yields a pipe dream $P'$ whose weight $\alpha$ satisfies the conditions of the lemma.
	\end{proof}
	
	For any diagram $D\subseteq [n]^2$, define the \emph{weight} of $D$ to be the vector $\mathrm{wt}(D)\in\mathbb{Z}^n$ whose $i$th component counts the number of boxes in row $i$ of $D$.
	
	\begin{definition}
		For any diagram $D$, the \emph{upper closure} $\overline{D}$ is the diagram 
		\[\overline{D} = \{(i,j)\ \mid \ j=j' \mbox{ and } i\leq i' \mbox{ for some } (i',j')\in D  \}. \]
	\end{definition}
	
	\begin{theorem}[{\cite[Theorem 1.2]{grothortho}}]
		\label{thm:upwardsdivisibility}
		For any permutation $w\in S_n$, 
		\[\mathrm{wt}(\overline{D(w)})\geq \alpha \mbox{ for all } \alpha\in \supp(\mathfrak{G}_w).\]		
		Consequently, 
		\[\deg \mathfrak{G}_w\leq \#\overline{D(w)}. \]
	\end{theorem}
	
	The following theorem gives a polytopal interpretation for Theorem \ref{thm:upwardsdivisibility}. 
	For the Minkowski sum below, we use the natural inclusions $\mathbb{R}^{k-1}\to \mathbb{R}^k$ by appending a zero. 
	\begin{theorem}
		\label{thm:superset}
		Let $w\in S_n$ be any permutation and let $D(w)$ have columns $D_1,\ldots,D_n$. Set $d_j = \max(D_j)$, taking $\max(\emptyset)=0$. Then,
		\[\mathrm{Newton}(\mathfrak{G}_w) \subseteq \sum_{j=1}^n P_{\mathrm{sp}}(\mathrm{SM}_{d_j}(D_j)).\]
	\end{theorem}
	\begin{proof}
		Let $\alpha\in\supp(\mathfrak{G}_w)$. By repeated use of Lemma \ref{lem:downwardsdivisibility}, we can write $\alpha = \beta+\gamma$ where $\beta\in\supp(\mathfrak{S}_w)$ and $\gamma\geq 0$. By Theorem \ref{thm:fms}, we can find a decomposition 
		\[\beta=\beta^{(1)}+\cdots+\beta^{(n)} \]
		where each $\beta^{(j)}$ is the indicator vector of a basis of $\mathrm{SM}_n(D_j)$.
		
		Let $\delta=\mathrm{wt}(\overline{D(w)})$, and note that
		\[\delta= \delta^{(1)}+\cdots+\delta^{(n)} \]
		where $\delta^{(j)}$ is the indicator vector in $\mathbb{R}^n$ of $\overline{D_j}=[d_j]$. Let $A$ be the $n\times n$ matrix with columns $\delta^{(j)}$. Equivalently, $A_{i,j}=1$ if and only if $(i,j)\in \overline{D(w)}$. 
		
		Observe that $\beta^{(j)}_i=1$ means $A_{ij}=1$. Call the entry $(i,j)$ of $A$ marked if $\beta^{(j)}_{i}=1$. Fix any $i\in[n]$. Since $\gamma_i\geq 0$, we have $\beta_i\leq \alpha_i$, so $\delta_i-\beta_i\geq \delta_i-\alpha_i$. Since $\delta_i-\beta_i$ is the number of unmarked entries of $A$ in row $i$, there are at least $\delta_i-\alpha_i$ unmarked entries in row $i$ of $A$. 
		
		For each $p\in [n]$, pick any $\delta_i-\alpha_i$ unmarked entries in row $i$ of $A$. Set all these entries to 0 to get a new matrix $B$. Let $\epsilon^{(j)}$ be the $j$th column of $B$ for each $j\in [n]$. By construction, 
		\[\epsilon^{(1)}+\cdots +\epsilon^{(n)} = \delta-(\delta-\alpha)=\alpha. \]
		From the use of marked entries, we see that $\delta^{(j)}\geq \epsilon^{(j)}\geq\beta^{(j)}$ for each $j$, so $\epsilon^{(j)} \in P_{\mathrm{sp}}(\mathrm{SM}_{d_j}(D_j))$.
		Thus,
		\[\alpha = \sum_{j=1}^n \epsilon^{(j)} \in \sum_{j=1}^n P_{\mathrm{sp}}(\mathrm{SM}_{d_j}(D_j)). \qedhere \]
	\end{proof}

	
	
	We now make two pairs of conjectures describing the support of Grothendieck polynomials. We provide partial results and describe some implications of the conjectures.
	 
	Theorem \ref{thm:upwardsdivisibility} shows the vector $\mathrm{wt}(\overline{D(w)})$
	is an upper bound (in $\mathbb{Z}^n$) for $\supp(\mathfrak{G}_w)$.
	When $\mathrm{wt}(\overline{D(w)})\in \supp(\mathfrak{G}_w)$, it is the unique maximal element. 
	We conjecture that all maximal elements of $\supp(\mathfrak{G}_w)$ have the same degree. 
		
	\begin{namedconjecture}[\ref{conj:1}]
		If $\alpha\in \mathrm{supp}(\mathfrak{G}_w)$ and $|\alpha|<\deg \mathfrak{G}_w$, then there exists $\beta\in\mathrm{supp}(\mathfrak{G}_w)$ with $\alpha< \beta$.
	\end{namedconjecture}
	The following is a natural strengthening of Conjecture \ref{conj:1}, dual to Lemma \ref{lem:downwardsdivisibility}.
	\begin{namedconjecture}[\ref{conj:2}]
		If $\alpha\in \mathrm{supp}(\mathfrak{G}_w)$ and $|\alpha|<\deg \mathfrak{G}_w$, then there exists $\beta\in\mathrm{supp}(\mathfrak{G}_w)$ with $\alpha< \beta$ and $|\beta|=|\alpha|+1$.
	\end{namedconjecture}
	We prove Conjecture \ref{conj:1} in the special case of fireworks permutations by confirming $\mathrm{wt}(\overline{D(w)})\in \supp(\mathfrak{G}_w)$.
	
	\begin{definition}[{\cite[Definition 3.5]{PSW}}]
		A permutation $w \in S_n$ is called \emph{fireworks} if the initial elements of its decreasing runs occur in increasing order.
	\end{definition}
	
	\begin{example}
		Consider $w=267419853$. The decreasing runs of $w$ are $2|6|741|9853$, so $w$ is fireworks since $2<6<7<9$. The Rothe diagram of $w$ is shown in Figure \ref{fig:fireworks-rothe-example}.
	\end{example}
	\begin{figure}[h]
		\includegraphics[scale=1]{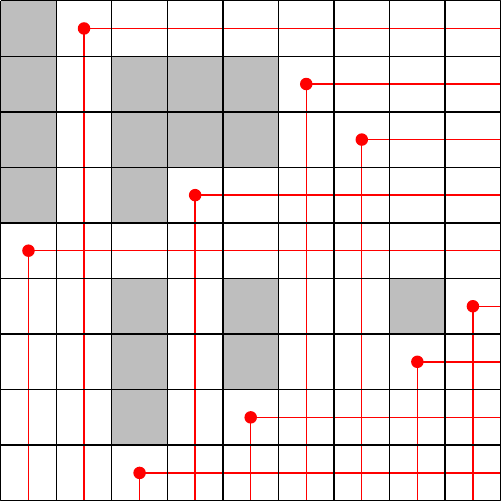}
		\caption{The Rothe diagram of the fireworks permutation $w=267419853$.}
		\label{fig:fireworks-rothe-example}
	\end{figure}
	Observe in the previous example that there is a dot directly below the southmost box in each column of $D(w)$. We show this property characterizes fireworks permutations.
	\begin{proposition}
		\label{prop:fireworks-diagram}
		Let $w\in S_n$ and $D(w)$ have columns $D_1,\ldots,D_n$. Then $w$ is fireworks if and only if $D_{w(j)}\neq \emptyset$ implies $\max(D_{w(j)}) = j-1$.
	\end{proposition}
	\begin{proof}
		First, note that $D_{w(j)}\neq \emptyset$ if and only if there is $i<j$ with $w(i)>w(j)$. In particular when $w(j-1)>w(j)$, the box $(j-1,w(j))$ is the southmost box in column $w(j)$ of $D(w)$. 
		
		Suppose $w$ is fireworks. If $w(j)$ is initial in a decreasing run of $w$, then $w(j)>w(i)$ for all $i<j$, so $D_{w(j)}=\emptyset$. If $w(j)$ is not initial in a decreasing run of $w$, then $w(j-1)>w(j)$ and we are done.
		
		Conversely, suppose $w$ is not fireworks. Then, we can find a decreasing run $w(i)>w(i+1)>\cdots>w(j-1)$ with $w(j-1)<w(j)<w(i)$. Then $D_{w(j)}\neq \emptyset$, but $(j-1,w(j))\notin D_{w(j)}$, so $\max(D_{w(j)})<j-1$.
	\end{proof}

	\begin{definition}[{\cite{PSW}}]
		The \emph{Rajchgot code} of $w\in S_n$ is the vector $\rajcode(w) = (r_1,\ldots,r_n)$, where $r_j$ is defined as follows.
		For each $j$, choose an increasing subsequence of $w(j),w(j+1), \ldots,w(n)$ containing $w(j)$ and of greatest length
		among all such subsequences. Let $r_j$ be the number of terms from $w(j),w(j+1),\ldots,w(n)$ omitted
		to form the chosen subsequence. 
	\end{definition}
	
	\begin{theorem}[{\cite[Theorem 1.1]{PSW}}]
		\label{thm:pswleadingterm}
		Let $w\in S_n$ and $\rajcode(w) = (r_1,\ldots,r_n)$. Then $\deg \mathfrak{G}_w=r_1+\cdots+r_n$, and in any term order satisfying $x_1<x_2<\cdots<x_n$, the leading term of $\mathfrak{G}_w$ is a scalar multiple of $x^{\rajcode(w)}$.
	\end{theorem}

	The following two lemmas describe $\rajcode(w)$ when $w$ is a fireworks permutation.
	\begin{lemma}
		\label{lem:rajcode-fireworks}
		Let $w \in S_n$ be fireworks. The Rajchgot code $\rajcode(w)=(r_1,\ldots,r_n)$ can be read off from $w$ as follows:
		\[
		r_i = 
		\begin{cases}
			0 &\textup{ if } i = n \\
			r_{i+1} &\textup{ if } w(i) < w(i+1)\\
			r_{i+1} + 1 &\textup{ if } w(i) > w(i+1)
		\end{cases}
		\]
	\end{lemma}
	\begin{proof}
		Since $w$ is fireworks, one greatest length increasing subsequence of $w(j),w(j+1),\ldots,w(n)$ (starting with $w(j)$) consists of $w(j)$ together with the initial elements of all subsequent decreasing runs.
	\end{proof}

	\begin{lemma}
		\label{lem:fireworks-raj-equals-closure}
		When $w\in S_n$ is fireworks, $\rajcode(w)=\mathrm{wt}(\overline{D(w)})$
	\end{lemma}
	\begin{proof}
		Let $\rajcode(w)=(r_1,\ldots,r_n)$. From Proposition \ref{prop:fireworks-diagram} and its proof together with Lemma \ref{lem:rajcode-fireworks} we see that
		\begin{align*}
			\mathrm{wt}(\overline{D(w)})_i &= \# \mbox{ of columns of $D(w)$ with a box in any row $j\geq i$}\\
			&=\#\{j\in [n] \mid w(j)\mbox{ is not initial in a decreasing run of $w$, and }j\geq i+1 \}\\
			&=r_i.\qedhere
		\end{align*}
	\end{proof}

	Denote by $\mathfrak G_w^\mathrm{top}$ the highest-degree nonzero homogeneous component of $\mathfrak G_w$. 	
	\begin{theorem}
		\label{thm:fireworks-support}
		Let $w\in S_n$ be fireworks. Then $\supp(\mathfrak G_w^\mathrm{top}) = \{\mathrm{wt}(\overline{D(w)}) \}$. 
	\end{theorem}
	\begin{proof}
		Theorem \ref{thm:upwardsdivisibility} establishes $\mathrm{wt}(\overline{D(w)})$ as an upper bound for $\supp(\mathfrak{G}_w)$ (under componentwise comparison). 
		Theorem \ref{thm:pswleadingterm} together with Lemma \ref{lem:fireworks-raj-equals-closure} show that when $w$ is fireworks, \[\mathrm{wt}(\overline{D(w)})=\rajcode(w) \in \supp(\mathfrak{G}_w^{\mathrm{top}}).\]
		
		Since no two elements in the support of a homogeneous polynomial can be componentwise comparable, the theorem follows.
	\end{proof}

	
	\begin{theorem}
		\label{thm:fireworks}
		Conjecture \ref{conj:1} holds for fireworks permutations. 
	\end{theorem}
	
	We now make one other pair of conjectures, further specifying the poset structure of $\supp(\mathfrak{G}_w)$.
	
	\begin{namedconjecture}[\ref{conj:3}]
		Fix any $w\in S_n$. If $\alpha,\gamma\in\mathrm{supp}(\mathfrak{G}_w)$, then
		\[\left\{\beta\in \mathbb{Z}^n \mid \alpha\leq \beta\leq \gamma \right\} \subseteq \mathrm{supp}(\mathfrak{G}_w).  \]
	\end{namedconjecture}
	
	\begin{namedconjecture}[\ref{conj:4}]
		For all $w\in S_n$, $\mathfrak{G}_w$ has SNP and $\mathrm{Newton}(\mathfrak{G}_w)$ is a generalized polymatroid.
	\end{namedconjecture}
	
	\begin{remark}
		The assertion that $\mathfrak{G}_w$ always has SNP is \cite[Conjecture 5.5]{MTY}. 
	\end{remark}
	We record two implications between Conjectures \ref{conj:1}--\ref{conj:4}:
	\begin{itemize}
		\item Conjecture \ref{conj:4} implies Conjecture \ref{conj:3} via the defining inequalities of generalized polymatroids (Definition \ref{def:gpolymatroid}). 
		\item Conjectures \ref{conj:1} and \ref{conj:3} together imply Conjecture \ref{conj:2}.
	\end{itemize}
	We also note that via a property of generalized polymatroids (\cite[Theorem 14.2.5]{frank}), Conjecture \ref{conj:4} is implied by \cite[Conjecture 22]{schlorentzian}. Conjecture \ref{conj:4} is a strengthening of \cite[Conjecture 5.1]{genpermflows}.
	
 	
 	We prove Conjectures \ref{conj:2} and \ref{conj:4} for Grassmannian permutations. We begin by reviewing the main result of \cite{symmgrothnewton}. Recall a permutation $w$ is \emph{Grassmannian} if $w$ has exactly one descent. It is well-known that Grassmannian permutations are in bijection with partitions: when $w$ has a descent at position $r$, the corresponding partition $\lambda$ is given by $\lambda=(w(r)-r,\cdots,w(2)-2,w(1)-1)$.
 	
 	Let $\lambda\in\mathbb{Z}^n$ be a partition and consider the Young diagram of $\lambda$ in English notation. Set $\mu^{(0)} = \lambda$. For $j\geq 1$, define $\mu^{(j)}$ to be $\mu^{(j-1)}$ with a box added to the northmost row $r$ such that the addition still yields a partition, and $\mu^{(j-1)}_r-\mu^{(0)}_r<r-1$. Stop when no such box exists. Let the resulting partitions be $\mu^{(0)},\ldots,\mu^{(N)}$. For Grassmannian $w\in S_n$ corresponding to $\lambda$, define $\mathrm{Par}(w)=\{\mu^{(0)},\ldots,\mu^{(N)}\}$, the partitions constructed from $\lambda$. Recall that \emph{dominance order} on partitions is defined by $\rho\trianglelefteq \nu$ if $\rho_1+\cdots+\rho_{i}\leq \nu_1+\cdots+\nu_i \mbox{ for all } i \mbox{ and } |\rho|=|\nu|$.
 	
 	\begin{example}
 		Let $\lambda=(5,5,1,1)$. Then we obtain the sequence of partitions
 		\begin{center}
 			\includegraphics{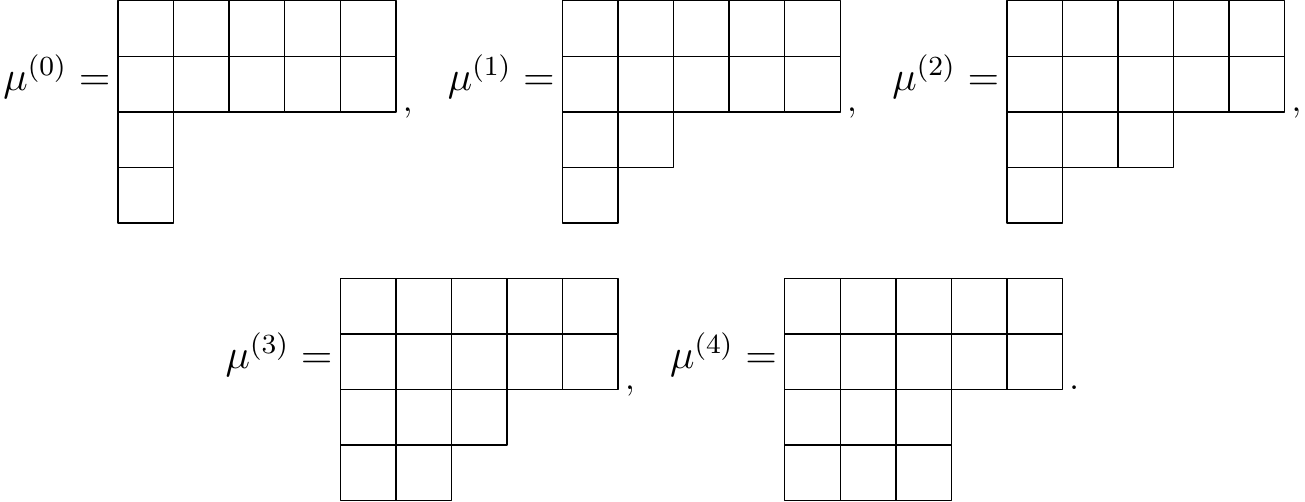}
 		\end{center}
 	\end{example}
 	
 	\begin{theorem}[{\cite{symmgrothnewton}}]
 		\label{thm:escobar-yong}
 		Suppose $w$ is a Grassmannian permutation with $\mathrm{Par}(w)=\{\mu^{(0)},\ldots,\mu^{(N)} \}$. Then $\deg\mathfrak{G}_w = |\mu^{(N)}|$, and for $0\leq j\leq N$, the support of the $(\ell(w)+j)$th degree homogeneous component of $\mathfrak{G}_w$ is exactly
 		\[\{\alpha\in\mathbb{Z}^n \mid \alpha\geq 0 \mbox{ and } \alpha\trianglelefteq\mu^{(j)}\}. \]
 		In particular, $\mathfrak{G}_w$ has SNP.
 	\end{theorem}
 
 	\begin{theorem}
 		\label{thm:grass}
 		Conjecture \ref{conj:4} and Conjecture \ref{conj:2} hold when $w$ is a Grassmannian permutation.
 	\end{theorem}
 	\begin{proof}
 		Let $\mathrm{Par}(w)=\{\mu^{(0)},\ldots,\mu^{(N)} \}$ with $\lambda=\mu^{(0)}$. We first confirm Conjecture \ref{conj:2}. Let $\alpha\in\supp(\mathfrak{G}_w)$ with $|\alpha|=\ell(w)+j<\deg \mathfrak{G}_w$. By Theorem \ref{thm:escobar-yong}, $\alpha\trianglelefteq \mu^{(j)}$. Let $e_i$ be the standard basis vector such that $\mu^{(j+1)}-\mu^{(j)} = e_i$. Then $\alpha+e_i\trianglelefteq \mu^{(j+1)}$, so $\alpha+e_i\in\supp(\mathfrak{G}_w)$. Hence $\beta=\alpha+e_i$ confirms Conjecture \ref{conj:2}.
 		
 		We now confirm Conjecture \ref{conj:4}. Define functions $y,z:2^{[n]}\to\mathbb{R}$ by
 		\[y(I) = \lambda_n + \cdots + \lambda_{n-\#I+1 }\quad \mbox{and}\quad z(I)=\mu^{(N)}_1+\cdots\mu^{(N)}_{\#I}.\]
 		It is straightforward to check that $(y,z)$ is a paramodular pair. Let $Q$ be the corresponding generalized polymatroid
 		\[Q=\left\{t\in \mathbb{R}^n \ \middle| \ y(I)\leq \sum_{i\in I} t_i\leq z(I) \mbox{ for all } I\subseteq [n]\right\}. \]
 		Observe that 
 		\[\min_{t\in\mathrm{Newton}(\mathfrak{G}_w)} \sum_{i\in I}t_i = \lambda_n + \cdots + \lambda_{n-\#I+1 } \quad \mbox{and}\quad \max_{t\in\mathrm{Newton}(\mathfrak{G}_w)} \sum_{i\in I}t_i = \mu^{(N)}_1+\cdots+\mu^{(N)}_{\#I},\]
 		so $\mathrm{Newton}(\mathfrak{G}_w) \subseteq Q$. We prove the opposite inclusion.
 		
 		Let $q\in Q\cap\mathbb{Z}^n$ and $|q| = |\lambda|+j$. We must show $q\trianglelefteq \mu^{(j)}$. We have
 		\[y(I)\leq \sum_{i\in I} q_i=\sum_{i=1}^n q_i-\sum_{i\notin I}q_i=|\lambda|+j -\sum_{i\notin I}q_i, \] 
 		so that
 		\[\sum_{i\notin I}q_i\leq \lambda_1+\cdots+\lambda_{n-\#I}+j =\lambda_1+\cdots+\lambda_{\#([n]\setminus I)} +j. \]
 		Replacing $[n]\setminus I$ with $I$, we see
 		\begin{equation}
 			\label{eq:1}
 			\sum_{i\in I}q_i\leq \lambda_1+\cdots+\lambda_{\#I} +j \mbox{ for all } I\subseteq [n].
 		\end{equation}
 		Set $\alpha$ to be the vector $\alpha=\mu^{(N)}-\lambda$, so 
 		\begin{equation}
 			\label{eq:2}
 			\sum_{i\in I}q_i\leq z(I)=(\lambda_1+\alpha_1)+\cdots+(\lambda_{\#I}+\alpha_{\#I}) \mbox{ for all } I\subseteq [n].
 		\end{equation}
 		
 		
 		Fix any $k\in [n]$. Note that $j\leq \alpha_1+\cdots+\alpha_k$. Then (\ref{eq:1}) shows
 		\[\mu^{(j)}_1+\cdots+\mu^{(j)}_k= \lambda_1+\cdots+\lambda_k+j\geq q_1+\cdots+q_k. \]
 		Hence, we have shown $q\trianglelefteq \mu^{(j)}$.
 	\end{proof}
	
	The following characterization of equality in Theorem \ref{thm:superset} would follow from Conjectures \ref{conj:1} and \ref{conj:4}.
	\begin{proposition}
		\label{prop:converse}
		Let $w\in S_n$ be any permutation and let $D(w)$ have columns $D_1,\ldots,D_n$. Set $d_j = \max(D_j)$, taking $\max(\emptyset)=0$. Assuming Conjectures \ref{conj:1} and \ref{conj:4} hold, it follows that 
		\[\deg \mathfrak{G}_w=\#\overline{D(w)} \mbox{ if and only if } \mathrm{Newton}(\mathfrak{G}_w) = \sum_{j=1}^n P_{\mathrm{sp}}(\mathrm{SM}_{d_j}(D_j)).\]
	\end{proposition}
	\begin{proof}
		Since, Conjectures \ref{conj:1} and \ref{conj:3} together imply Conjecture \ref{conj:2}, we are assuming Conjecture \ref{conj:2} holds as well. Suppose first that the polyhedral equality holds. By linearity,
		\[\mathrm{Newton}(\mathfrak{G}_w^{\mathrm{top}}) = \{\mathrm{wt}(\overline{D(w)})\}. \]
		Thus $\deg\mathfrak{G}_w=\#\overline{D(w)}$.
		
		Conversely, assume $\deg\mathfrak{G}_w = \#\overline{D(w)}$. By Theorem \ref{thm:upwardsdivisibility}, we have $\mathrm{Newton}(\mathfrak{G}_w^{\mathrm{top}}) = \{\mathrm{wt}(\overline{D(w)})\}$. 
		Since we are assuming Conjecture \ref{conj:4}, $\mathfrak{G}_w$ has SNP and $\mathrm{Newton}(\mathfrak{G}_w)$ is a generalized polymatroid. 
				
		Set $Q=\sum_{j=1}^n P_{\mathrm{sp}}(\mathrm{SM}_{d_j}(D_j))$. By Propositions \ref{prop:mink-sum} and \ref{prop:matroidgps}, $Q$ is a generalized polymatroid. 
		Denote its associated paramodular pair by $(y,z)$. Observe that the integer points of $P_{\mathrm{sp}}(\mathrm{SM}_{d_j}(D_j))$ satisfy the conclusions of Lemma \ref{lem:downwardsdivisibility} and Conjecture \ref{conj:2}. Then, Lemmas \ref{lem:recovery} and \ref{lem:integralgpoly} imply
		\begin{align*}
		y(I) &= \min\left\{\sum_{i\in I}q_i\ \ \middle|\ q\in Q\cap \mathbb{Z}^n \right\} 
		  =\min\left\{\sum_{i\in I}q_i\ \ \middle|\ q\in \mathbb{Z}^n \cap \sum_{j=1}^nP(\mathrm{SM}_n(D_j)) \right\}\\
		  &=\min\left\{\sum_{i\in I}q_i\ \ \middle|\ q\in \mathbb{Z}^n \cap \mathrm{Newton}(\mathfrak{S}_w)\right\}
		  =\min\left\{\sum_{i\in I}q_i\ \ \middle|\ q\in \mathrm{Newton}(\mathfrak{G}_w) \right\}.
		\end{align*}
		Similarly, Lemmas \ref{lem:recovery} and \ref{lem:integralgpoly} together with Conjecture \ref{conj:2} imply
		\begin{align*}
		z(I) &= \max\left\{\sum_{i\in I}q_i\ \ \middle|\ q\in Q\cap \mathbb{Z}^n\right\}
		   =\max\left\{\sum_{i\in I}q_i\ \ \middle|\ q=\mathrm{wt}(\overline{D(w)}) \right\}\\
		   &=\max\left\{\sum_{i\in I}q_i\ \ \middle|\ q\in \mathbb{Z}^n\cap \mathrm{Newton}(\mathfrak{G}_w^{\mathrm{top}}) \right\}
		   =\max\left\{\sum_{i\in I}q_i\ \ \middle|\ q\in \mathrm{Newton}(\mathfrak{G}_w) \right\}.
		\end{align*}
		Thus, the paramodular pairs of $Q$ and $\mathrm{Newton}(\mathfrak{G}_w)$ coincide. Consequently, $Q=\mathrm{Newton}(\mathfrak{G}_w)$.
	\end{proof}

	\begin{remark} 
		We note that all of our conjectures (\ref{conj:1}--\ref{conj:coeff}) naturally lift to double Grothendieck polynomials, and appear to hold there as well.	
	\end{remark}

	\section{Coefficients and pricipal specialization of Grothendieck polynomials}
	\label{sec:mobius}
	For each permutation $w$, we describe a poset $P_w$. For certain permutations, we connect the M\"obius function of $P_w$ to the coefficients of the Grothendieck polynomial $\mathfrak{G}_w$. Recall the \emph{M\"obius function} $\mu$ of a finite poset $P$ is the unique function $P\times P\to \mathbb{Z}$ defined by 
	\[\mu(p,q) = \begin{cases}
		0 &\mbox{ if } p\nleq q,\\
		1 &\mbox{ if } p=q,\\
		-\displaystyle\sum_{p\leq r<q}\mu(p,r)&\mbox{ if }p<q.
	\end{cases} \]
	\begin{definition}
		\label{def:poset}
		For each $w\in S_n$, define $P_w$ to be the poset 
		\[\{\beta\in\mathbb{Z}^n \mid \alpha \leq \beta \leq \mathrm{wt}(\overline{D(w)}) \mbox{ for some } \alpha\in \supp(\mathfrak{G}_w)\},\]
		together with a minimum element denoted $\hat{0}$.
	\end{definition}
	
	\begin{namedconjecture}[\ref{conj:mobius}]
		Let $w$ be a permutation such that all nonzero coefficients of $\mathfrak{S}_w$ equal 1. If $\mu_{w}$ is the M\"obius function of $P_w$, then
		\[\mathfrak{G}_w=-\sum_{\beta\in P_w-\hat{0}}\mu_w(\hat{0}, \beta)x^{\beta}.\]
	\end{namedconjecture}
	We have tested Conjecture \ref{conj:mobius} for all permutations $w\in S_8$.

	\begin{example}
		Set $w=351624$. Then $\mathrm{wt}(\overline{D(w)}) = (3,3,2,2)$, and 
		\begin{align*}
		\mathfrak{G}_w = &\phantom{a}
		(x_1^3x_2^2 x_3^2 
		+x_1^2x_2^3 x_3^2 
		+x_1^3x_2^3 x_3 
		+x_1^3x_2^2 x_3 x_4 
		+x_1^2x_2^3 x_3 x_4
		+x_1^3x_2^3 x_4 
		+x_1^2x_2^3 x_4^2 
		+x_1^3x_2^2 x_4^2 
		)\\
		&-(2 x_1^3x_2^3 x_3^2 
		+x_1^3x_2^2 x_3^2 x_4 
		+x_1^2x_2^3 x_3^2 x_4
		+3 x_1^3x_2^3 x_3 x_4
		+x_1^2x_2^3 x_3 x_4^2
		+x_1^3x_2^2 x_3 x_4^2
		+2 x_1^3x_2^3 x_4^2
		)\\
		&+(2 x_1^3x_2^3 x_3^2 x_4
		+2 x_1^3x_2^3 x_3 x_4^2 
		).
		\end{align*}
		The poset $P_w$ and its M\"obius function $\mu_{w}$ are shown in Figure \ref{fig:zeroonemobius}.
	\end{example}
	
	\begin{figure}[h]
		\begin{center}
			\includegraphics[scale=.85]{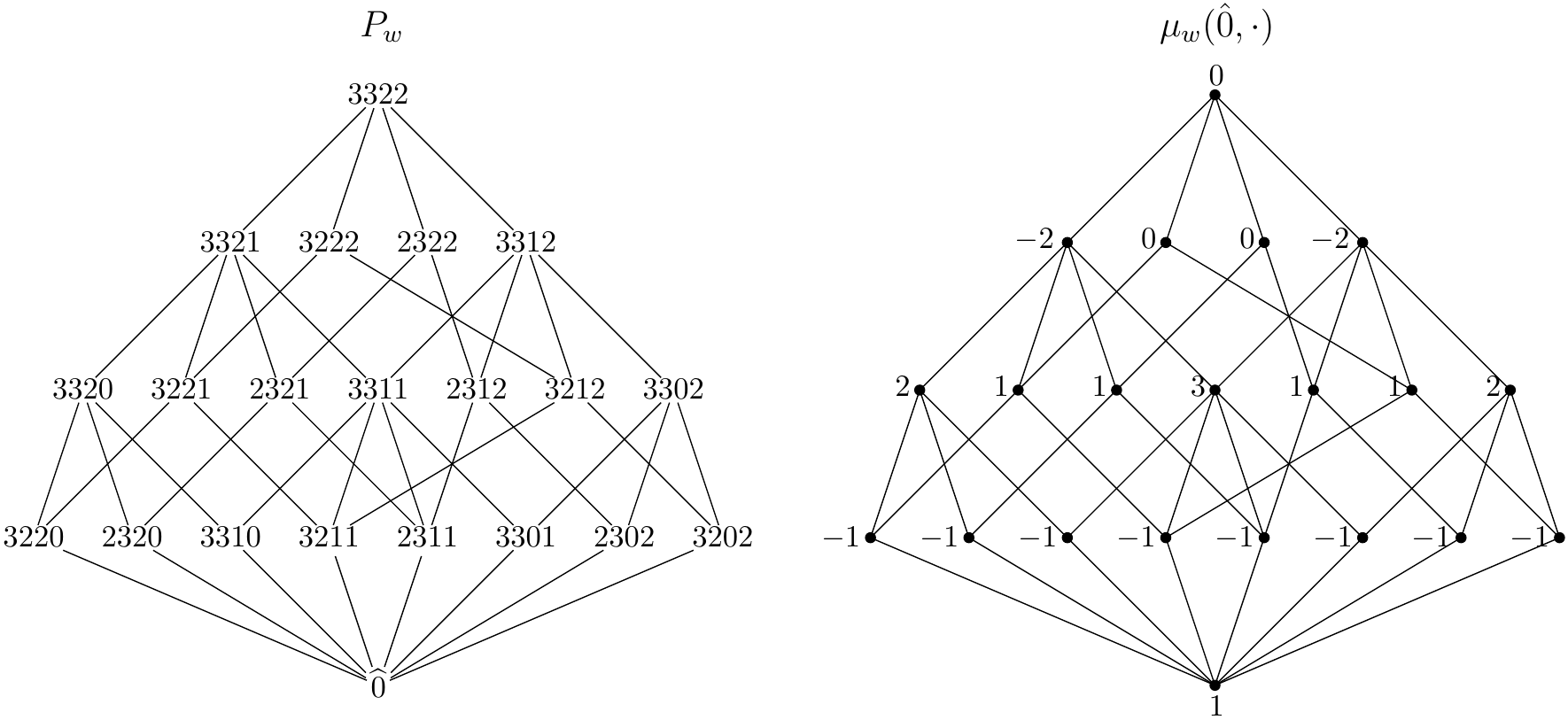}
			\caption{}
			\label{fig:zeroonemobius}
		\end{center}
	\end{figure}

 	
 	The class of permutations covered by Conjecture \ref{conj:mobius} is characterized by the following result.
 	\begin{theorem}[{\cite[Theorem 4.8]{zeroone}}]
 		A permutation $w$ avoids the patterns 12543, 13254, 13524, 13542, 21543, 125364, 125634, 215364, 215634, 315264, 315624,
 		and 315642
 		if and only if all nonzero coefficients of $\mathfrak{S}_w$ equal 1.
 	\end{theorem}
	
	We conjecture one last property connecting the poset structure of $\supp(\mathfrak{G}_w)$ to the coefficients of $\mathfrak{G}_w$.
	\begin{namedconjecture} [\ref{conj:coeff}]
		Fix $w\in S_n$ and let $\mathfrak{G}_w = \sum_{\alpha\in\mathbb{Z}^n}C_{w\alpha}x^\alpha$. For any $\beta\in\supp(\mathfrak{G}_w^{\mathrm{top}})$,
		\[\sum_{\alpha\leq \beta} C_{w\alpha} = 1. \]
	\end{namedconjecture}
	
	When $\supp(\mathfrak{G}_w)$ has a unique maximum element (such as for fireworks permutations),
	Conjecture \ref{conj:coeff} specializes to
	\[\mathfrak{G}_w(1,\ldots,1)=1.\] 
	The principal specialization of $\beta$-Grothendieck polynomials was previously considered by Kirillov \cite[Comment 3.2]{dunkl} as well as by the first author of this paper \cite{belong}. While we assume that $\mathfrak{G}_w(1,\ldots,1)=1$ for all $w \in S_n$ is known to experts, we did not find a proof of this fact in the literature. We include a short proof below based on pipe dream complexes.

	We recall basic facts about the Euler characteristic of simplicial complexes. Let $\Delta$ be any simplicial complex. Denote the interior faces of $\Delta$ by $\mathrm{int}(\Delta)$, and the boundary faces by $\mathrm{bd}(\Delta)$. Suppose $\Delta$ has $f_i$ faces of dimension $i$ for each $i\geq 0$. Recall the \emph{Euler characteristic} $\chi(\Delta)$ is the alternating sum
	\[\chi(\Delta) = f_0-f_1+f_2-f_3+\cdots \]
	It is well-known that $\chi(\Delta)=1$ when $\Delta$ is a ball, and that $\chi(\Delta)=1+(-1)^{d-1}$ when $\Delta$ is the boundary of a $d$-dimensional ball. Recall the pipe dream complex $\Delta_w$ of $w\in S_n$ (see Theorem \ref{thm:pdcomplex} and Example \ref{exp:pdcomplex}).
	
	\begin{proposition}
		\label{prop:grothprinspec}
		For any permutation $w\in S_n$, 
		\[\mathfrak{G}_w(1,\ldots, 1) = 1.\]
	\end{proposition}
	\begin{proof}
		Let $w\in S_n$. If $w=w_0$, then clearly the result holds. Otherwise, Theorem \ref{thm:ball} implies $\Delta_w$ is a ball of dimension $d=\binom{n}{2}-\ell(w)-1$. Thus
		\[\chi(\mathrm{int}(\Delta_w)) = \chi(\Delta_w) - \chi(\mathrm{bd}(\Delta_w)) = 1-(1+(-1)^{d-1}) = (-1)^d. \]
		From Theorems \ref{thm:pdformula} and \ref{thm:pdcomplex}, one observes $\mathfrak{G}_w(1,\ldots,1) = (-1)^d\chi(\mathrm{int}(\Delta_w))$. Hence $\mathfrak{G}_w(1,\ldots,1) = 1.$
	\end{proof}
	Alternatively one can deduce the preceding result from \cite[Lemma 2.3]{belong}. Theorem \ref{thm:fireworks-support} yields the following.
	
	\begin{corollary}
		Conjecture \ref{conj:coeff} holds for fireworks permutations.
	\end{corollary}
	
	\section*{Acknowledgments} 
	We are grateful to Allen Knutson, Vic Reiner, Ed Swartz, and Alexander Yong for helpful discussions and feedback.
	
	\bibliographystyle{plain}
	\bibliography{grothendieck-bibliography}
\end{document}